\documentclass[11pt]{amsart}
\usepackage[utf8]{inputenc}
\usepackage{amssymb,fullpage,ulem}
\usepackage{amsthm}
\usepackage{amsmath}
\usepackage{tikz}
\usepackage{cite}
\usepackage{hhline} 
\usepackage[math]{cellspace}
\usepackage{graphicx}
\usetikzlibrary{math}
\usepackage{color, colortbl,xcolor,hyperref}
\usepackage{multicol}

\theoremstyle{definition} 
\newtheorem{theorem}{Theorem}[section]
\newtheorem{example}{Example}[section]

\newtheorem{rmk}[theorem]{Remark}
\newtheorem{lemma}[theorem]{Lemma}
\theoremstyle{definition}
\newtheorem{definition}{Definition}[section]
\theoremstyle{remark}
\newtheorem*{remark}{Remark}

\def\hroot{\tilde{\alpha}}
\def\a{\alpha}

\def\v{\textbf{v}}
\def\x{\textbf{x}}

\definecolor{LightCyan}{rgb}{0.88,1,1}

\newcommand{\supp}[1]{\mathrm{supp}#1}

\title{Broken Bracelets and Kostant's Partition Function}

\author{Mark Curiel}
\address[M.\ Curiel]{Department of Mathematics\\University of Hawai`i M\={a}noa
}
\email[M.\ Curiel]{\textcolor{blue}{\href{mailto:curielm@hawaii.edu}{curielm@hawaii.edu}}}

\author{Elizabeth Gross}
\address[E.\ Gross]{Department of Mathematics\\University of Hawai`i M\={a}noa
}
\email{\textcolor{blue}{\href{mailto:egross@hawaii.edu}{egross@hawaii.edu}}}

\author{Pamela E. Harris}
\address[P.\ E.\ Harris]{Department of Mathematics and Statistics\\Williams College\\
\url{https://www.pamelaeharris.com/}}
\email[P.~E.~Harris]{\textcolor{blue}{\href{mailto:peh2@williams.edu}{peh2@williams.edu}}}

\newcommand{\nbinom}[2]{\binom{#1}{#2}_{\mathfrak{R}}}

\DeclareMathOperator{\Star}{Star}

\newenvironment{brsm}{\left[\begin{smallmatrix}}{\end{smallmatrix}\right]}

\makeatletter
\newcommand*{\encircled}[1]{\relax\ifmmode\mathpalette\@encircled@math{#1}\else\@encircled{#1}\fi}
\newcommand*{\@encircled@math}[2]{\@encircled{$\m@th#1#2$}}
\newcommand*{\@encircled}[1]{%
  \tikz[baseline,anchor=base]{\node[draw,circle,outer sep=0pt,inner sep=.2ex] {#1};}}
\makeatother

\begin{document}

\maketitle

\begin{abstract}
    Inspired by the work of Amdeberhan, Can, and Moll on broken necklaces, we define a broken bracelet as a linear arrangement of marked and unmarked vertices and introduce a generalization called $n$-stars, which is a collection of $n$ broken bracelets whose final (unmarked) vertices are identified. 
    Through these combinatorial objects, we provide a new  framework  for  the  study  of  Kostant's  partition  function, which counts the number of ways to express a vector as a nonnegative integer linear combination of the positive roots of a Lie algebra.  
    Our main result establishes that (up to reflection) the number of broken bracelets with a fixed number of unmarked vertices with nonconsecutive marked vertices gives an upper bound for the value of Kostant's partition function for multiples of the highest root of a Lie algebra of type $A$. We connect this work to multiplex juggling sequences, as studied by Benedetti, Hanusa, Harris, Morales, and Simpson, by providing a correspondence to an equivalence relation on $n$-stars.
\end{abstract}

\section{Introduction}
A vector partition function problem can be stated as follows:
Let $A$ be an $m\times d$ integral matrix. Then for ${\bf{b}}$ in the nonnegative linear span of the columns of $A$, we want to compute 
the vector partition function
\begin{equation} \label{eq:vectorpartitionfunction}
    \wp_A({\bf{b}})=\#\{{\bf{x}}\in\mathbb{N}^d\;:\;A{\bf{x}}={\bf{b}}.\}
\end{equation}
In other words, $\wp_A(\bf{b})$ gives the number of ways one can express $\bf{b}$ as a nonnegative integral linear combination of the columns of $A$ and each way we can do this is referred to as a partition of $\bf{b}$.

The study of vector partition functions arises within many contexts in mathematics. Their study appears in the literature in the context of number theory  via the study of integer partitions, which is the special case where $A$ is an $1 \times n$ matrix and the entries of $A$ are the integers allowed as parts in the partitions.
Other contexts where vector partitions arise include commutative algebra  via the study of Hilbert series \cite{miller2004combinatorial},
algebraic geometry via toric varieties \cite{fulton2016introduction}, algebraic statistics via goodness-of-fit testing for log-linear models \cite{diaconis1998algebraic}, discrete geometry via the study of integer lattice point enumeration in polyhedra \cite{barvinok1994polynomial} \cite{DELOERA20041273}, and optimization via integer programming\cite{de2019algebraic}. 

A vector partition function also arises within the context of representation theory of Lie  algebras. In setting, the vector partition function is known  as Kostant's partition function and it counts the number of ways to express a weight $\mu$ of a simple Lie algebra $\mathfrak{g}$ as a nonnegative integer linear combination of the positive roots of $\mathfrak{g}$. This partition function appears as the (signed) terms in the computation of weight multiplicities for irreducible representations of classical simple Lie algebras via the use of Kostant's weight multiplicity formula \cite{Kostant}. 
We remark that Kostant's partition function also arises in the context of polyhedral geometry via the enumeration of integer lattice points of flow polytopes \cite{flow,flow2}.

Finding closed formulas for the value of Kostant's partition function remains a very active field of study and recent work has connected vector partitions to multiplex juggling sequences \cite{Juggling}. Inspired by that work, we provide a new framework for the study of Kostant's partition function via a connection to a new set of combinatorial objects called \textit{broken bracelets}, defined following the conventions of Amdeberhan, Can, and Moll \cite{ACM},
and their generalizations, which we call \textit{$n$-stars}. 
In this setting, a broken bracelet is a sequence of marked and unmarked vertices where no two marked vertices are adjacent and reflections are identified. 
We first consider the vector defined as twice the highest root of a Lie algebra of type $A_r$, which we denote by $2\tilde\alpha$, and our results in Section \ref{sec:background} provide a way to associate broken bracelets to vector partitions of $2\tilde\alpha$ using the positive roots of a Lie algebra of type $A_r$. 
From this we establish that the number of broken bracelets gives an upper bound to the number of vector partitions $2\tilde\alpha$.

With these initial results at hand, in Section \ref{sec:special case}, we give an equivalence relation on broken bracelets, which we extend to $n$-stars in Section \ref{sec:main}. 
From this we can establish that the equivalence classes, arising from this relation, are in bijection with the set of partitions of $n\tilde\alpha$ for all $n\geq 2$. A key insight in this work is the fact that the equivalence relation we define extends the notion of reflection on broken bracelets. 
The paper culminates, in Section \ref{sec: juggling}, by providing a connection between our equivalence relation on $n$-stars and multiplex juggling sequences as studied by Benedetti, Hanusa, Harris, Morales, and Simpson in \cite{Juggling}.

\section{Background}\label{sec:background}
In this section we begin by setting some notation regarding our main objects of study: Kostant's partition function and broken bracelets.

\subsection{Kostant's partition function}
Kostant's partition function is a vector partition function in which the columns defining matrix $A$ in Equation \eqref{eq:vectorpartitionfunction} are the positive roots of a simple Lie algebra. 
In this work, we specialize to the Lie algebra of type $A_r$ with $r$ being a positive integer. In this case, we let $e_i \in \mathbb{R}^{r+1}$ be the $i$th standard unit vector, and then, following the notation in \cite{GW}, the set of simple roots is $\Delta=\{\a_i=e_i-e_{i+1}:1\leq i\leq r\}$ and the set of positive roots is given by $\Phi^{+}=\Delta\cup \{\a_i+\cdots+\a_{j}=e_i-e_{j+1}:1\leq i<j\leq r\}$. 
Thus, the matrix $A$ has as its columns the set $\Phi^+$, and we let $\wp_A(\textbf{b})$ count the number of ways to express the vector $\textbf{b}$ as nonnegative integral linear combination of the elements in $\Phi^+$. 
As this is the matrix we utilize throughout, to simplify our notation, we drop the subscript and let the partition function be denoted by $\wp(\textbf{b})$.
The main object of interest in our study is the partitions of the highest root of the Lie algebra of type $A_r$, which is defined as $\hroot=\a_1+\cdots+\a_{r}=e_1-e_{r+1}$.  As an example, consider $r=2$. Then $\wp(2\hroot)=3$ as we could write $2\hroot=2\a_1+2\a_2$ as a sum of the elements in the multisets $\{\alpha_1+\alpha_2,\alpha_1+\alpha_2\}$,  $\{\a_1,\a_2,\a_1+\a_2\}$, or $\{\a_1,\a_1,\a_2,\a_2\}$.

We now introduce an additional set of vectors in order to simplify our  computations, as well as to illustrate the connection between vector partitions and broken bracelets.
For integers $1\leq i \leq j\leq r$, let $E_{ij} = \sum_{k=i}^j e_k$ and define $V_r = \{E_{ij} : 1 \leq i \leq j \leq r\}$. Note that $E_{ii} = e_i$, so in this case we keep the simpler index and write $E_i$ instead.
Given a vector $\v\in\mathbb{R}^r$, if $\v$ can be expressed as a nonnegative integer linear combination of the vectors in $V_r$, we call such an expression (up to reordering of terms) a $V_r$-combination; equivalently, a $V_r$-combination is a multiset with ground set $V_r$. From this we can define a new vector partition function, which we denote $P:\mathbb{R}^r\to\mathbb{N}$, such that $P(\v)$ gives the number of ways to express $\v$ as a $V_r$-combination.

For the remainder of this note, let ${\bf{x}} = \sum_{i=1}^r e_i=\sum_{i=1}^r E_i$. 
Now we consider the case when $r=2$, and note that $P(2\x)=2$ since $2\x$ can be written as the  $V_2$-combination where we take the sums of the elements in the multisets $\{E_{12},E_{12}\}$, $\{E_{1},E_{2},E_{12}\}$, or $\{E_1,E_1,E_2,E_2\}$. It is no coincidence that this is precisely the value of $\wp(2\hroot)$. 
In fact, for any $n\in\mathbb{N}$ the number of ways to express $n{\bf{x}}$ as a $V_r$-combination is precisely the value of $\wp(n\hroot)$ using the positive of the Lie algebra of type $A_{r-1}$. This is our first result.

\begin{lemma}
If $r\geq 1$, then $\wp(n\hroot)=P(n\text{\x})$.
\end{lemma}
\begin{proof}
It suffices to show that there is a bijection between the set of partitions of $n\hroot$ using the positive roots of the Lie algebra of type $A_{r-1}$ and the set of $V_r$-combinations of the vector $n\x$.

We begin by establishing that the map $\gamma:V_r\to\Phi^+$ defined by 
$E_i\mapsto \a_i$ and $E_{ij}\mapsto\a_i+\cdots+\a_j$ is a bijection. 
We note that $\gamma$ is onto since for any $\a_i$, where $1\leq i\leq r$, there exists $E_i\in V_r$ such that $\gamma(E_i)=\a_i$  and, additionally, for any $\a_i+\cdots+\a_j$, with $1\leq i<j\leq r$, there exists $E_{ij}\in V_r$ such that $\gamma(E_{ij})=\a_i+\cdots+\a_j$. 
To establish injectivity, consider $E_{ij}$ (with $1\leq i\leq j\leq r$) and $E_{mn}$ (with $1\leq m\leq n\leq r$) in $V_r$ such that $\gamma(E_{ij})=\gamma(E_{mn})$. 
If $i< j$ and $m< n$, then $\gamma(E_{ij})=\gamma(E_{mn})$ implies that $\a_i+\cdots+\a_j=\a_m+\cdots+\a_n$. 
Thus, $i=m$ and $j=n$ since the sums are consecutive in the index and hence $E_{ij}=E_{mn}$. 
In the case that $i=j$ and $m\leq n$, then $\gamma(E_{ij})=\gamma(E_{mn})$ implies that $\a_i=\a_m+\cdots+\a_n$. 
Thus $m=n=i$ and again $E_{ij}=E_{mn}$. The remaining case where $i\leq j$ and $m=n$ is analogous. 
Therefore $\gamma$ is a bijection. 

The importance of $\gamma$ being a bijection is that we can now extend this bijection to any multiset consisting of the vectors used in the partitions of $n\hroot$ and take them to a $V_r$-combination of $n\x$. This new map is also invertible, and hence it implies that $\wp(n\hroot)=P(n\text{\x})$. 
\end{proof}

\begin{rmk}\label{rmk:notation} To highlight the importance of the consecutive sums of vectors in the set $V_r$, it will be convenient to make the following identifications. If $\supp(e_i) = \{i\}$, then $\supp(E_{ij}) = \{i,\ldots,j\}$. Thus, we identify {$E_i$ with the integer $i$ and} $E_{ij}$ with the consecutive sequence of integers from $i$ to~$j$. For instance, $E_{24} = 234$ and $E_5 = 5$. In this way, we identify $V_r$-combinations with consecutive sequences of integers separated by a dash. For example, if $r=4$, then $2E_{1} + E_{23} + E_{24} + E_{4}$ is identified with 1-1-23-234-4. The identification is not unique since 23-1-4-1-234 represents the same $V_r$-combination of $2{\bf{x}}$.
\end{rmk}

\begin{lemma}\label{split}
Let $r$ be a positive integer and let $2{\bf{x}} = \sum_{v \in V_r} c_v v$ where $c_v \in \mathbb{Z}_{\geq0}$ for each $v \in V_r$. Then there exist subsets $V_1,V_2 \subseteq V_r$ and nonnegative integers $a_v,b_v$ for each $v\in V_r$ such that ${\bf{x}} = \sum_{v \in V_1} a_vv = \sum_{v \in V_2} b_vv$ with $a_v + b_v = c_v$ for each $v \in V_r$. 
\end{lemma}

\begin{proof} 
Let $r$ be a positive integer and let $2{\bf{x}} = \sum_{v \in V_r} c_v v$ where $c_v \in \mathbb{Z}_{\geq0}$. Then for each $1 \leq i \leq r$ either (i) there are exactly two vectors $v_i$ and $v_i'$ in $V_r$ such that $i \in \supp(v_i)$, $i \in \supp(v_i')$, and $c_{v_i}=c_{v_i'}=1$ or (ii) there is a single vector $v_i$ such that $i \in \supp(v_i)$, and $c_{v_i}=2$. If we let $v_i'=v_i$ when we are in case (ii), then
 we have a list of $2r$ (not necessarily distinct) vectors $v_1, v_1',\ldots,v_r,v_r'$ of $V_r$ such that $i$ belongs to the support of both $v_i$ and $v_i'$. Let $V_1 = \{v_1\}$ and $V_2 = \{v_1'\}$ and set $i=2$.
\begin{enumerate}
    \item If $v_i,v_i' \not\in V_1 \cup V_2$, then let $v_i \in V_1$ and $v_i' \in V_2$. Set $i=i+1$ and repeat this step.
    
    \item If $v_i \not\in V_1 \cup V_2$ and $v_i' \in V_1$ (or $v_i' \in V_2$), then let $v_i \in V_2$ ($v_i \in V_1$). Set $i=i+1$ and go to step one.
    
    \item If $v_i' \not\in V_1 \cup V_2$ and $v_i \in V_1$ (or $v_i \in V_2$), then let $v_i' \in V_2$ ($v_i' \in V_1$). Set $i=i+1$ and go to step one.
    
    \item If $v_i,v_i' \in V_1 \cup V_2$, then set $i=i+1$ and go to step one.
\end{enumerate}
We have a finite list of vectors so this process must terminate. The result is two subsets $V_1,V_2$ of $V_r$ with the desired properties.
\end{proof}

The importance of Lemma~\ref{split} is that it allows us to think about a $V_r$-combination of $2{\bf{x}}$ as a sum of two $V_r$-combinations of ${\bf{x}}$. In fact, this observation was a key insight in the proof of \cite[Corollary 3.9]{Juggling} by Steve Butler which established that the value of $\wp(n\hroot)$ is the same as the number of multiplex juggling sequences of length $r$ which start and end in configuration $\langle n\rangle$. We will say more about this in Section \ref{sec: juggling}. 

\subsection{Broken Bracelets}
Broken necklaces were introduced by Amdeberhan, Can, and Moll in \cite{ACM}. Here we use broken bracelet to refer to a specific type of broken necklace that is discussed, but left unnamed, in \cite{ACM}.

\begin{definition}
A \textit{broken bracelet} is a linear arrangement of $n$ vertices $v_1 \cdots v_n$ where $k$ vertices are non-consecutively marked. Let $g_k(n)$ be the number of broken bracelets up to reflection.
\end{definition}

Two bracelets $b = v_1 \cdots v_n$ and $b' = u_1 \cdots u_n$ are \textit{equal} if for each $i$, $v_i$ and $u_i$ are both marked or both unmarked; in this case, we write $b'=b$.

\begin{example}
Up to reflection, there are four broken bracelets with $n=5$ and $k=2$, and we illustrate them in Figure \ref{fig:g_2(5)} indicating marked vertices with red.
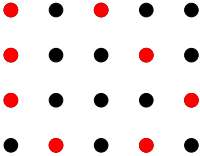
\begin{figure}[h]
    \centering
    \begin{tikzpicture}[scale=.6]
    \tikzmath{\c=.15;}
\foreach \n in {0,...,4}{
    \draw[fill=black] (\n,0) circle (\c);
}
\draw[red,fill=red] (0,0) circle (\c);
\draw[red,fill=red] (2,0) circle (\c);
\foreach \n in {0,...,4}{
    \draw[fill=black] (\n,-1) circle (\c);
}
\draw[red,fill=red] (0,-1) circle (\c);
\draw[red,fill=red] (3,-1) circle (\c);

\foreach \n in {0,...,4}{
    \draw[fill=black] (\n,-2) circle (\c);
}
\draw[red,fill=red] (0,-2) circle (\c);
\draw[red,fill=red] (4,-2) circle (\c);
\foreach \n in {0,...,4}{
    \draw[fill=black] (\n,-3) circle (\c);
}
\draw[red,fill=red] (1,-3) circle (\c);
\draw[red,fill=red] (3,-3) circle (\c);
\end{tikzpicture}
    \caption{The four  broken bracelets with $n=5$ and $k=2$, culminating in $g_2(5)=4$.}
    \label{fig:g_2(5)}
\end{figure}
\end{example}

Broken bracelets encode $V_r$-combinations, and we begin by illustrating this encoding with an example.

\begin{example}\label{ex:broken}  In order to arrive at a correspondence from bracelets to $V_r$-combinations, there needs to be a way to associate the vertices within a bracelet with vectors in a $V_r$-combination. Our goal is to associate the consecutive unmarked vertices with a single vector and have the marked vertices denote the separation between vectors.
Since, in this association, a broken bracelet will encode a $V_r$-combination of the vector 2\x,  Lemma \ref{split} implies the  broken bracelet encodes \textit{two} $V_r$-combinations of \x, and thus we need a mechanism to split a bracelet into two parts, one for each $V_r$-combination of \x.  Also note
that broken bracelets may begin or end in a marked vertex, in which case we need to decide how to handle such marked vertices. 
Before we give the technical process addressing these two issues, we illustrate the process below. 

Consider the broken bracelet in Figure~\ref{fig:broken bracelet} that has length $n=12$ and $k= 5$ marked vertices; this broken bracelet  encodes a $V_{5}$-combination of $2\x$ with $k+2 = 7$ parts.

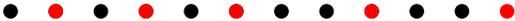
\begin{figure}[h]
    \centering
    \begin{tikzpicture}[scale = .6] 
    \tikzmath{\c=.15;}
\foreach \n in {0,...,10}{
\draw[fill=black] (\n,0) circle (\c);
}
\draw[red,fill=red] (1,0) circle (\c);
\draw[red,fill=red] (3,0) circle (\c);
\draw[red,fill=red] (5,0) circle (\c);
\draw[red,fill=red] (8,0) circle (\c);
\draw[red,fill=red] (11,0) circle (\c);
\end{tikzpicture}
    \caption{Broken bracelet for Example \ref{ex:broken}.}
    \label{fig:broken bracelet}
\end{figure}

To get to the $V_5$-combination, first append an unmarked vertex at the beginning and end of the bracelet, and color blue the ``central'' unmarked vertex, i.e.\ the unmarked vertex lying in the middle of the bracelet. The new bracelet we obtain is: \begin{tikzpicture}[scale = .5] 
\tikzmath{\d=.12;}
\foreach \n in {-1,...,12}{
\draw[fill=black] (\n,0) circle (\d);
}
\draw[red,fill=red] (1,0) circle (\d);
\draw[red,fill=red] (3,0) circle (\d);
\draw[red,fill=red] (5,0) circle (\d);
\draw[blue,fill=blue] (6,0) circle (\d);
\draw[red,fill=red] (8,0) circle (\d);
\draw[red,fill=red] (11,0) circle (\d);
\end{tikzpicture} . Next, we see that the central unmarked vertex is the eighth vertex and here is where we split the bracelet into two, keeping the central unmarked vertex at the end and at the beginning of the two new bracelets, respectively: \begin{tikzpicture}[scale = .5] 
\tikzmath{\d=.12;}
\foreach \n in {-1,...,6}{
\draw[fill=black] (\n,0) circle (\d);
}
\draw[red,fill=red] (1,0) circle (\d);
\draw[red,fill=red] (3,0) circle (\d);
\draw[red,fill=red] (5,0) circle (\d);
\draw[blue,fill=blue] (6,0) circle (\d);
\end{tikzpicture} \, and \, \begin{tikzpicture}[scale = .5] 
\tikzmath{\d=.12;}
\foreach \n in {6,...,12}{
\draw[fill=black] (\n,0) circle (\d);
}
\draw[blue,fill=blue] (6,0) circle (\d);
\draw[red,fill=red] (8,0) circle (\d);
\draw[red,fill=red] (11,0) circle (\d);
\end{tikzpicture} .  Now, for the former bracelet, each unmarked vertex gets an integer 1 though 5, increasing from left to right. And, for the latter bracelet, each vertex gets an integer 1 through 5, decreasing from left to right. Lastly, each marked vertex is associated to a dash. Thus, we get the $V_5$-combinations 12-3-4-5 and 54-32-1 of {\bf x} that when combined into 12-3-4-5-1-23-45 results in a $V_5$-combination of 2{\bf x}. Pictorially, we may represent this as follows, where the top bracelet is the original and the bottom is the modified bracelet in which we have added a marked vertex to denote where we split the bracelet:

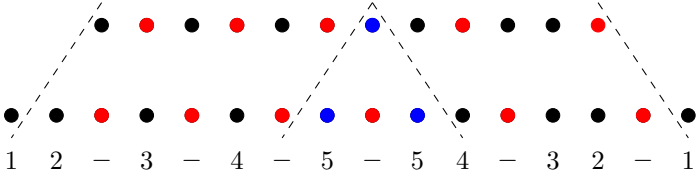
\[
\begin{tikzpicture}[scale=.6]
\tikzmath{\c=.15;}
\foreach \n in {0,...,10}{
\draw[fill=black] (\n,0) circle (\c);
}
\draw[red,fill=red] (1,0) circle (\c);
\draw[red,fill=red] (3,0) circle (\c);
\draw[red,fill=red] (5,0) circle (\c);
\draw[red,fill=red] (8,0) circle (\c);
\draw[red,fill=red] (11,0) circle (\c);
\draw[blue,fill=blue] (6,0) circle (\c);

\foreach \n in {-2,...,13}{
\draw[fill=black] (\n,-2) circle (\c);
}
\draw[red,fill=red] (0,-2) circle (\c);
\draw[red,fill=red] (2,-2) circle (\c);
\draw[red,fill=red] (4,-2) circle (\c);
\draw[red,fill=red] (6,-2) circle (\c);
\draw[red,fill=red] (9,-2) circle (\c);
\draw[red,fill=red] (12,-2) circle (\c);
\draw[blue,fill=blue] (5,-2) circle (\c);
\draw[blue,fill=blue] (7,-2) circle (\c);
\draw[dashed] (6,.5) -- (4,-2.5);
\draw[dashed] (6,.5) -- (8,-2.5);
\draw[dashed] (0,.5) -- (-2,-2.5);
\draw[dashed] (11,.5) -- (13,-2.5);
\draw (-2,-3) node {\small $1$};
\draw (-1,-3) node {\small $2$};
\draw (0,-3) node {\small $-$};
\draw (1,-3) node {\small $3$};
\draw (2,-3) node {\small $-$};
\draw (3,-3) node {\small $4$};
\draw (4,-3) node {\small $-$};
\draw (5,-3) node {\small $5$};
\draw (6,-3) node {\small $-$};
\draw (7,-3) node {\small $5$};
\draw (8,-3) node {\small $4$};
\draw (9,-3) node {\small $-$};
\draw (10,-3) node {\small $3$};
\draw (11,-3) node {\small $2$};
\draw (12,-3) node {\small $-$};
\draw (13,-3) node {\small $1$};
\end{tikzpicture}
\]

Using this process, we are now able to establish the following theorem.

\end{example}

\begin{theorem}\label{upperbound}
Let $\x$ have length $r\geq 2$. Then $P(2\x)\leq \displaystyle\sum_{m=2}^{2r} g_{m-2}(2r+m-5)$.
\end{theorem}
\begin{proof}
We can consider the process described in Example \ref{ex:broken} as a map from the set of broken bracelets with $2r+m-5$ vertices and $m-2$ marked vertices up to reflection to the set of $V_r$-combinations consisting of exactly $m$ vectors. This map is well-defined since if we reflect a broken bracelet and apply the process described, we get the same $V_r$-combination. Furthermore, this map is surjective.  Let $\sum_{v \in V_r} c_v v$ be a $V_r$-combination of $2{\bf{x}}$ with $\sum_{v \in V_r} c_v = m$.  By Lemma \ref{split}, there are subsets $V_1$ and $V_2$ such that ${\bf{x}} = \sum_{v \in V_1} a_vv = \sum_{v \in V_2} b_vv$ with $a_v + b_v = c_v$ for each $v \in V_r$. Each of the sets $V_1$ and $V_2$ give a $V_r$-combination of ${\bf{x}}$. Write these $V_r$-combinations of ${\bf{x}}$ increasingly using the notation of Remark \ref{rmk:notation}.  If we work backwards as described in Example \ref{ex:broken} above, the result is a broken bracelet of length $2r+m-5$ with $m-2$ marked vertices. Under our map, this broken bracelet maps onto the $V_r$-combination that we began with, verifying the claim of surjectivity. Now, if we consider similar maps as we vary $m$ then we get a sequence of surjective maps (where the domain of each map is the set of broken bracelets of length $2r+m-5$ with $m-2$ marked vertices up to reflection and the range is the the set of $V_r$-combinations with $m$ parts). The domain of each surjective map has size $g_{m-2}(2r+m-5)$. The result follows.
\end{proof}

Theorem \ref{upperbound} gives an upper bound for the number of partitions of 2{\bf x} in terms of the numbers $g_k(n)$ of certain broken bracelets up to reflection. In \cite{ACM}, the authors provide a way to enumerate $g_k(n)$ from smaller broken bracelets, specifically, 
\[
g_k(n) = g_{k}(n-2) + g_{k-2}(n-4) + \binom{n-k-1}{k-1}.
\]
Due to the constraint that the marked vertices are nonconsecutive, the authors translate the numbers $g_n(k)$ into \textit{necklace binomial coefficients}, which are defined as $\nbinom{t}{k} = g_k(t+k-1)$. The triangle of necklace binomial coefficients is known as \textit{Losanitsch's triangle} \cite{Losanitch}; the first ten rows of this triangle are shown in the top table of Figure \ref{fig:triangles}. In the paper \cite{ACM}, the authors provide closed-form solutions to certain necklace binomial coefficients: 
\begin{align*}
    \nbinom{k}{k} = 1, \quad & \nbinom{k+1}{k} = \left\lfloor \frac{k+2}{2} \right\rfloor, \quad \nbinom{k+2}{k} = \left\lfloor \frac{(k+2)^2}{4} \right\rfloor, \; \text{and}\\
    \nbinom{k+3}{k} &= \sum_{j=0}^k (-1)^{k-j} \left( \sum_{i=0}^j \left\lfloor \frac{j+2}{2} \right\rfloor + \binom{j+1}{2} \right).
\end{align*}

\noindent Moreover, the necklace binomial coefficients are symmetric in that sense that $\nbinom{t}{k} = \nbinom{t}{t-k}$. Thus, the above four formulas give the first four columns of the top triangle in Figure \ref{fig:triangles}, respectively. The table beneath Losanitch's triangle in Figure \ref{fig:triangles} is a triangle that we include for comparison which enumerates the number of partitions of $2\x$ with $m$ parts and where $\x$ has length $r$. For ease of comparison, the rows of interest in the Losanitsch's triangle are highlighted blue. The highlighted rows of Losanitch's triangle serve an upperbound to the rows of the triangle below it. Secifically, the correspondence for the two triangles in Figure \ref{fig:triangles} is the following: the $(t,k)$-entry of Losanitsch's triangle is an upper bound for the $(\frac{t}{2}+1,k+2)$-entry of the triangle below it. 

\begin{figure}[h]
    \centering
    \resizebox{8.5cm}{!}{%
    \begin{tabular}[scale=.75]{|c||c|c|c|c|c|c|c|c|c|c|c|}
    \hhline{-||-----------}
    \textbf{t/k} & \textbf{0} & \textbf{1} & \textbf{2} & \textbf{3} & \textbf{4} & \textbf{5} & \textbf{6} & \textbf{7} & \textbf{8} & \textbf{9} & \textbf{10}\\
    \hhline{=::===========}
    \textbf{1} & \textbf{1} & \textbf{1} &  &  &  &  &  &  &  &  &\\
    \rowcolor{LightCyan}
    \textbf{2} & \textbf{1} & \textbf{1} & \textbf{1} &  & &  &  &  &  &  &\\
    \textbf{3} & \textbf{1} & \textbf{2} & \textbf{2} & \textbf{1} &  &  &  &  &  &  &\\
    \rowcolor{LightCyan}
    \textbf{4} & \textbf{1} & \textbf{2} & \textbf{4} & \textbf{2} & \textbf{1} &  &  &  &  &  &\\
    \textbf{5} & \textbf{1} & \textbf{3} & \textbf{6} & \textbf{6} & \textbf{3} & \textbf{1} &  &  &  &  & \\
    \rowcolor{LightCyan}
    \textbf{6} & \textbf{1} & \textbf{3} & \textbf{9} & \textbf{10} & \textbf{9} & \textbf{3} & \textbf{1} &  &  &  & \\
    \textbf{7} & \textbf{1} & \textbf{4} & \textbf{12} & \textbf{19} & \textbf{19} & \textbf{12} & \textbf{4} & \textbf{1} &  &  & \\
    \rowcolor{LightCyan}
    \textbf{8} & \textbf{1} & \textbf{4} & \textbf{16} & \textbf{28} & \textbf{38} & \textbf{28} & \textbf{16} & \textbf{4} & \textbf{1} &  & \\
    \textbf{9} & \textbf{1} & \textbf{5} & \textbf{20} & \textbf{44} & \textbf{66} & \textbf{66} & \textbf{44} & \textbf{20} & \textbf{5} & \textbf{1} & \\
    \rowcolor{LightCyan}
    \textbf{10} & \textbf{1} & \textbf{5} & \textbf{25} & \textbf{60} & \textbf{110} & \textbf{126} & \textbf{110} & \textbf{60} & \textbf{25} & \textbf{5} & \textbf{1}\\
    \hhline{-||-----------}
\end{tabular}}

\vspace{.5cm}
    \resizebox{8.5cm}{!}{%
\begin{tabular}[scale=.75]{|c||c|c|c|c|c|c|c|c|c|c|c|}
    \hhline{-||-----------}
    \textbf{r/m} & \textbf{2} & \textbf{3} & \textbf{4} & \textbf{5} & \textbf{6} & \textbf{7} & \textbf{8} & \textbf{9} & \textbf{10} & \textbf{11} & \textbf{12}\\
    \hhline{=::===========}
    \textbf{2} & \textbf{1} & \textbf{1} & \textbf{1} &  &  &  &  &  &  &  &\\
    \textbf{3} & \textbf{1} & \textbf{2} & \textbf{4} & \textbf{2} & \textbf{1} &  &  &  &  &  &\\
    \textbf{4} & \textbf{1} & \textbf{3} & \textbf{9} & \textbf{10} & \textbf{8} & \textbf{3} & \textbf{1} &  &  &  &\\
    \textbf{5} & \textbf{1} & \textbf{4} & \textbf{16} & \textbf{28} & \textbf{34} & \textbf{24} & \textbf{13} & \textbf{4} & \textbf{1} &  &\\
    \textbf{6} & \textbf{1} & \textbf{5} & \textbf{25} & \textbf{60} & \textbf{100} & \textbf{106} & \textbf{83} & \textbf{45} & \textbf{19} & \textbf{5} & \textbf{1}\\
    \hhline{-||-----------}
\end{tabular}}
    \caption{Losanitsch's triangle is above and below is a triangle enumerating bracelets of length $2r+m-5$ with $m-2$ marked vertices.}
    \label{fig:triangles}
\end{figure}

\section{Refined enumeration}\label{sec:special case}

In the previous section we gave an upper bound for the total number of $V_r$-combinations of the vector $2\bf{x}$. The reason why this count is an upper bound rather than a formula is because there are cases where different bracelets encode the same $V_r$-combination. Moving forward, our strategy is to understand exactly when different bracelets encode the same $V_r$-combination.  

The following example is an instance of two bracelets that encode the same $V_4$-combination.  An approach to counting the number of distinct combinations is to set-up an equivalence relation on the set of broken bracelets that takes into account that swapping parts doesn't change a $V_r$-combination. 

\begin{example}\label{equivb4bracelets}
In the set-up from the previous section, equivalent $V_4$-combinations 1-2-3-4-12-34 and 1-2-34-12-3-4 are represented by the following different broken bracelets:
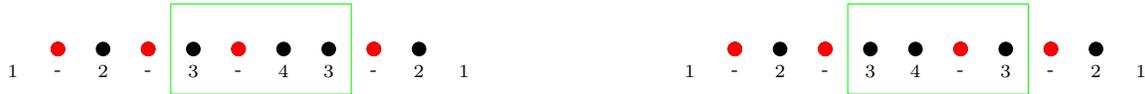
\[
\begin{tikzpicture}[scale = .6] 
\tikzmath{\c=.15;}
\foreach \n in {0,...,8}{
\draw[fill=black] (\n,0) circle (\c);
\draw[fill=black] (\n + 15,0) circle (\c);
}
\draw[red,fill=red] (0,0) circle (\c);
\draw[red,fill=red] (2,0) circle (\c);
\draw[red,fill=red] (4,0) circle (\c);
\draw[red,fill=red] (7,0) circle (\c);
\draw[red,fill=red] (15,0) circle (\c);
\draw[red,fill=red] (17,0) circle (\c);
\draw[red,fill=red] (20,0) circle (\c);
\draw[red,fill=red] (22,0) circle (\c);
\draw (-1,-.5) node {\tiny 1};
\draw (0,-.5) node {\tiny -};
\draw (1,-.5) node {\tiny 2};
\draw (2,-.5) node {\tiny -};
\draw (3,-.5) node {\tiny 3};
\draw (4,-.5) node {\tiny -};
\draw (5,-.5) node {\tiny 4};
\draw (6,-.5) node {\tiny 3};
\draw (7,-.5) node {\tiny -};
\draw (8,-.5) node {\tiny 2};
\draw (9,-.5) node {\tiny 1};

\draw (14,-.5) node {\tiny 1};
\draw (15,-.5) node {\tiny -};
\draw (16,-.5) node {\tiny 2};
\draw (17,-.5) node {\tiny -};
\draw (18,-.5) node {\tiny 3};
\draw (19,-.5) node {\tiny 4};
\draw (20,-.5) node {\tiny -};
\draw (21,-.5) node {\tiny 3};
\draw (22,-.5) node {\tiny -};
\draw (23,-.5) node {\tiny 2};
\draw (24,-.5) node {\tiny 1};
\draw[green] (2.5,-1) rectangle (6.5,1);
\draw[green] (17.5,-1) rectangle (21.5,1);
\end{tikzpicture}
\]
Notice that the bracelets are not the same if we read forward or backward, that is, they are different bracelets up to reflection. However, under a more careful inspection, we see that after reflecting the patterns boxed in green, we can obtain one bracelet from the other while maintaining the same $V_r$-combination. This reflection is indeed a swap of the ``34'' and ``3-4'' patterns in the $V_r$-combination. 
\end{example}

We now give some definitions that will help us talk about swapping more precisely. A \textit{pattern} is a connected sub-bracelet enclosed by marked vertices. 
\begin{definition}
Let $v_1v_2\cdots v_n$ be a broken bracelet.  The sequence $P = v_iv_{i+1} \cdots v_j$ is a \textit{pattern} of the bracelet $v_1v_2\cdots v_n$ if 
\begin{enumerate}
\item both vertices $v_{i-1}$ and $v_{j+1}$ are marked, \emph{or}
\item the vertex $v_{j+1}$ is marked and $i=1$, \emph{or} 
\item  the vertex $v_{i-1}$ is marked and $j=n$. 
\end{enumerate}
\end{definition}

Using the language above, every broken bracelet has the form $P_1 | P_2 | \cdots | P_k$ for some patterns $P_i$ and where the vertical lines represent marked vertices. In this way, we can easily talk about swapping patterns without specific reference to the supports of the vectors in a $V_r$-combination.

Before the next definition, it will be convenient to introduce the following notation. We will denote the reverse of a pattern $P$ by $\bar{P}$. For instance, in Example \ref{equivb4bracelets}, each boxed pattern is the reverse of the other. We also make the distinction between bracelets that are ``equal'' and ``equal up to reflection.''

\begin{definition}
Let $b$ be a broken bracelet. A sub-bracelet $b'$ of $b$ is a  \textit{central sub-bracelet} if $b'$ is a pattern and the number of unmarked vertices in $b$ to the right of $b'$ equals the number of unmarked vertices to the left of $b'$. 
\end{definition}

\begin{definition}
Let $b$ be a broken bracelet. A broken bracelet $b'$ is said to be a \textit{central reversal} of $b$ if it can be obtained by a reversal of a central sub-bracelet of $b$, i.e., if there exists a central sub-bracelet $Q$ of $b = P|Q|R$ such that $b' = P|\bar{Q}|R$.
\end{definition}

\begin{example}\label{fourpartcomb}
For this example, we revisit the bracelet in Example \ref{equivb4bracelets}, however, this time we use the notation we just introduced. This bracelet has the form $P|Q|R$ where $Q$ is the pattern boxed in green in Example \ref{equivb4bracelets}.

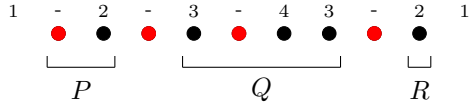
\begin{center}
\begin{tikzpicture}[scale = .6] 
\foreach \n in {0,...,8}{
\draw[fill=black] (\n,0) circle (.15);
}
\draw[red,fill=red] (0,0) circle (.15);
\draw[red,fill=red] (2,0) circle (.15);
\draw[red,fill=red] (4,0) circle (.15);
\draw[red,fill=red] (7,0) circle (.15);
\draw (-1,.5) node {\tiny 1};
\draw (0,.5) node {\tiny -};
\draw (1,.5) node {\tiny 2};
\draw (2,.5) node {\tiny -};
\draw (3,.5) node {\tiny 3};
\draw (4,.5) node {\tiny -};
\draw (5,.5) node {\tiny 4};
\draw (6,.5) node {\tiny 3};
\draw (7,.5) node {\tiny -};
\draw (8,.5) node {\tiny 2};
\draw (9,.5) node {\tiny 1};

\draw (-.25,-.5)--(-.25,-.75);
\draw (-.25,-.75)--(1.25,-.75);
\draw (1.25,-.75)--(1.25,-.5);

\draw (2.75,-.5)--(2.75,-.75);
\draw (2.75,-.75)--(6.25,-.75);
\draw (6.25,-.75)--(6.25,-.5);

\draw (7.75,-.5)--(7.75,-.75);
\draw (7.75,-.75)--(8.25,-.75);
\draw (8.25,-.75)--(8.25,-.5);

\filldraw (.5,-1.25) node {\small$P$};
\filldraw (4.5,-1.25) node {\small$Q$};
\filldraw (8,-1.25) node {\small$R$};
\end{tikzpicture}
\end{center}

\noindent Note $Q$ is a central sub-bracelet since $Q$ is indeed a pattern and there is one unmarked vertex to the left and right of $Q$. Furthermore, if we reverse the pattern $Q$ as in Example \ref{equivb4bracelets}, the result is a central reversal of the above bracelet. 

\end{example}

For bracelets $b$ and $b'$, we shall write $b \sim b'$ if there is a sequence of bracelets $b = b_0, \ldots, b_n = b'$ such that $b_{i+1}$ is a central reversal of $b_i$ for each $0\leq i\leq n-1$. We immediately have $b \sim \bar b$ since the reflection of a whole bracelet is a central reversal where $Q = b$. For a given $r$ and $m$, let $X_{r,m}$ be the set of bracelets of length $2r+m-5$ with $m-2$ marked vertices and let $X_{r,m}/\sim$ be the quotient space of $X_{r,m}$ by $\sim$.

\begin{theorem}\label{countbracelets}
Let ${\bf x}$ have length $r$. Then $P(2{\bf x}) = \sum_{m=2}^{2r} |X_{r,m}/\sim|$.
\end{theorem}

\begin{proof}
Similar to the proof of Theorem \ref{upperbound}, we can consider the process described in Example \ref{ex:broken} as a map from  $X_{r,m}/\sim$ to the set of $V_r$-combinations consisting of exactly $m$ vectors. This map is well-defined since central reversals appear in the $V_r$-combination as swapping the order of the vectors listed, which does not change the $V_r$-combination. Using arguments similar to those in the proof of Theorem \ref{upperbound}, this map is also surjective.

We now show the map is injective, that is, bracelets that encode the same partition of $2\x$ belong to the same class of bracelets. Let $b$ and $b'$ be bracelets in $X_{r,m}$ which correspond to the same partition of $2\x$. Each bracelet encodes a list of vectors, and by Lemma \ref{split}, each list of vectors can be grouped into two, say into sets $V_1$ and $V_2$ for $b$ and $U_1$ and $U_2$ for $b'$, so that the sum of vectors in each group is a partition of $\x$. For each $i = 1, 2$, let $V_i = \{v_{i1}, \ldots, v_{ik_i}\}$ where $k_i \in \mathbb{N}$, $v_{ij} \in \mathbb{Z}^n$ and $j = 1, \ldots, k_i$. Similarly, for each $i = 1, 2$, let $U_i = \{u_{i1}, \ldots, u_{i\ell_i}\}$ where $\ell_i \in \mathbb{N}$, $u_{ij} \in \mathbb{Z}^n$ and $j = 1, \ldots, \ell_i$. Without loss of generality, assume the sets $V_i$ and $U_i$ are written using the lexicographic order for each $i = 1,2$. Note that if $V_1 = U_1$, then necessarily $V_2 = U_2$ since $V_1 \cup V_2 = U_1 \cup U_2$ and each $V_i$, $U_i$ is a partition of $\x$. In this case, not only is $b \sim b'$ but in fact $b = b'$.

Now we will assume $V_1 \neq U_1$ and we will perform sequence of swaps between $V_1$ and $V_2$ based on central rotations such that at the end of the sequence $V_1$ and $U_1$ will contain the same vectors. Since $V_1$ and $U_1$ are written using the lexicographic order and $b$ and $b'$ encode the same partition of $2\x$, then $v_{i1} = u_{11}$ for some $i$.  If $i=1$, then we will do nothing, $v_{11}=u_{11}$ which is desired, else, we will perform the following operation. Since $b$ is considered a central sub-bracelet of $b$, then $\bar{b}$ is a central reversal of $b$. Thus, if we relabel $V_1,V_2$ and their elements to correspond to $\bar{b}$, the result is $v_{11} = u_{11}$. At this point, if $V_1 = U_1$, then $b' = \bar{b}$ and $b' \sim b$, as desired. Otherwise, assume for some $t > 1$, there is a sequence of bracelets $b = b_0, b_1, \ldots, b_{t-1}$ obtained by central reversals or making no change (i.e. we could have $b_{i+1}=b_i$) such that if the sets $V_1$, $V_2$ correspond to the bracelet $b_{t-1}$, then $v_{1j} = u_{1j}$ for all $j = 1, \ldots, t-1$. Since $b$ and $b'$ encode the same partition, then $v_{ij} = u_{1t}$ for some $i,j$. Since $v_{11} + \cdots + v_{1 (t-1)} = u_{11} + \cdots + u_{1 (t-1)}$, then $v_{1t} + \cdots + v_{1k_1} = u_{1t} + \cdots + u_{1\ell_1}$. Since $v_{ij} = u_{1t}$, then $v_{ij} + \cdots + v_{ik_i} = u_{1t} + \cdots + u_{1\ell_1}$. 
Hence, the vertices of the bracelet $b_{t-1}$ corresponding to the vectors $v_{1t} + \cdots + v_{1k_1} = v_{ij} + \cdots + v_{ik_i}$, so $\{v_{1t}, \ldots, v_{1k_1}\} \cup \{v_{ij}, \ldots, v_{ik_i}\}$ is a central sub-bracelet of $b_{t-1}$. Reversing this central sub-bracelet gives a new bracelet $b_t$ and if we relabel the sets $V_1$, $V_2$, and their elements to correspond to $b_t$ then $v_{1j} = u_{1j}$ for $j = 1,\ldots,t$. Repeat these steps until $t=\ell_1$, the result is $V_1$ and $U_1$ contain the vectors as desired.
\end{proof}

\section{Generalizing bracelets to {\it{n}}-stars}\label{sec:main}

Here we generalize the notion of a bracelet and a central sub-bracelet in an effort to count partitions of the vector $n{\bf x}$ for each positive integer $n \in \mathbb{N}$. 

\begin{definition}
A \textit{strand} is a sequence $v_1 \cdots v_\ell$ of marked and unmarked vertices such that no two marked vertices are consecutive and the final vertex $v_\ell$ is unmarked.
\end{definition}

\begin{example}\label{ex:strand} For a fixed $r$ a strand encodes a partition of ${\bf x}$, which we recall is simply the all ones vector. Figure \ref{fig:strand} below depicts a strand of length 8 which encodes the four-part partition of $\x=(1,1,1,1,1,1)$ on the right.

\begin{figure}[h]
\centering
\begin{tikzpicture}[scale = .6] 
\foreach \n in {0,...,7}{
\draw[fill=black] (\n,0) circle (.15);
}
\draw[red,fill=red] (1,0) circle (.15);
\draw[red,fill=red] (4,0) circle (.15);
\draw[red,fill=red] (6,0) circle (.15);
\draw (-1,-.5) node {\tiny 1};
\draw (0,-.5) node {\tiny 2};
\draw (1,-.5) node {\tiny -};
\draw (2,-.5) node {\tiny 3};
\draw (3,-.5) node {\tiny 4};
\draw (4,-.5) node {\tiny -};
\draw (5,-.5) node {\tiny 5};
\draw (6,-.5) node {\tiny -};
\draw (7,-.5) node {\tiny 6};

\draw (9,0) node {$\leadsto$};

\draw (17.5,0) node {
$E_{12} + E_{34} + E_{5} + E_{6} = \begin{brsm}
1\\1\\0\\0\\0\\0
\end{brsm}
+
\begin{brsm}
0\\0\\1\\1\\0\\0
\end{brsm}
+
\begin{brsm}
0\\0\\0\\0\\1\\0
\end{brsm}
+
\begin{brsm}
0\\0\\0\\0\\0\\1
\end{brsm}$
};
\end{tikzpicture}
    \caption{A decoding process for a strand with 5 marked vertices into a partition with 5-1=4 parts.}
    \label{fig:strand}
\end{figure}
\end{example}

\begin{definition}
A \textit{closing substrand} of a strand $v_1 \cdots v_\ell$ is either the entire strand $v_1 \cdots v_\ell$ or is obtained by deleting the initial $i$ vertices where $v_i$ is marked, namely $v_{i+1} \cdots v_\ell$. 
\end{definition}

For instance, the strand from Figure~\ref{fig:strand} has four closing substands and each correspond to a linear combination on the right.

\begin{figure}[h]
\centering
\begin{tikzpicture}[scale = .6] 
\foreach \n in {0,...,7}{
\draw[fill=black] (\n,0) circle (.15);
}
\draw[red,fill=red] (1,0) circle (.15);
\draw[red,fill=red] (4,0) circle (.15);
\draw[red,fill=red] (6,0) circle (.15);
\draw (-1,-.5) node {\tiny 1};
\draw (0,-.5) node {\tiny 2};
\draw (1,-.5) node {\tiny -};
\draw (2,-.5) node {\tiny 3};
\draw (3,-.5) node {\tiny 4};
\draw (4,-.5) node {\tiny -};
\draw (5,-.5) node {\tiny 5};
\draw (6,-.5) node {\tiny -};
\draw (7,-.5) node {\tiny 6};

\foreach \n in {2,...,7}{
\draw[fill=black] (\n,-2) circle (.15);
}
\draw[red,fill=red] (4,-2) circle (.15);
\draw[red,fill=red] (6,-2) circle (.15);
\draw (2,-2.5) node {\tiny 3};
\draw (3,-2.5) node {\tiny 4};
\draw (4,-2.5) node {\tiny -};
\draw (5,-2.5) node {\tiny 5};
\draw (6,-2.5) node {\tiny -};
\draw (7,-2.5) node {\tiny 6};

\foreach \n in {5,...,7}{
\draw[fill=black] (\n,-4) circle (.15);
}
\draw[red,fill=red] (6,-4) circle (.15);
\draw (5,-4.5) node {\tiny 5};
\draw (6,-4.5) node {\tiny -};
\draw (7,-4.5) node {\tiny 6};

\draw[fill=black] (7,-6) circle (.15);
\draw (7,-6.5) node {\tiny 6};

\draw (12,-0) node {$\leadsto$};
\draw (12,-2) node {$\leadsto$};
\draw (12,-4) node {$\leadsto$};
\draw (12,-6) node {$\leadsto$};

\draw (19,0) node {
$E_{12} + E_{34} + E_{5} + E_{6}$
};
\draw (19,-2) node {
$E_{34} + E_{5} + E_{6}$
};
\draw (19,-4) node {
$E_{5} + E_{6}$
};
\draw (19,-6) node {
$E_{6}$
};
\end{tikzpicture}
    \caption{The four closing substrands of the strand in Figure \ref{fig:strand} and their corresponding subpartitions.}
    \label{fig:substrand}
\end{figure}

\begin{definition}
An \textit{$n$-star} is the identification of $n$ strands at their final vertex. Let $\Star_n(r)$ be the set of all $n$-stars whose strands have $r-1$ unmarked vertices.
\end{definition}

\begin{remark}
A broken bracelet is a 2-star.
\end{remark}

\begin{example}\label{ex:star} A 3-star with $r=4$ with labeled strands is pictured in Figure \ref{fig:3star}; its strands are:

\[
\begin{array}{crcl}
    1. & \begin{tikzpicture}[scale=.6]
    \foreach \n in {4,...,9}{
    \draw[fill=black] (\n,0) circle (.15);
    }
    \draw[red,fill=red] (4,0) circle (.15);
    \draw[red,fill=red] (6,0) circle (.15);
    \draw[red,fill=red] (8,0) circle (.15);
    \end{tikzpicture} & \rightarrow & 1-2-3-4\\
    2. & \begin{tikzpicture}[scale=.6]
    \foreach \n in {3,...,7}{
    \draw[fill=black] (\n,0) circle (.15);
    }
    \draw[red,fill=red] (3,0) circle (.15);
    \draw[red,fill=red] (5,0) circle (.15);
    \end{tikzpicture} & \rightarrow & 1-2-34\\
    3. & \begin{tikzpicture}[scale=.6]
    \foreach \n in {3,...,6}{
    \draw[fill=black] (\n,0) circle (.15);
    }
    \draw[red,fill=red] (4,0) circle (.15);
    \end{tikzpicture} & \rightarrow & 12-34
\end{array}
\]
After identifying the above three strands at their final vertex, the result is the 3-star shown on the left in Figure \ref{fig:3star}. 
\vspace{.5cm}

\begin{figure}[h]
    \centering
    \begin{multicols}{3}
    \begin{tikzpicture}[scale=.55]
    \tikzmath{\a = 60;\b = 60;\c=.15;}
    \foreach \n in {4,...,9}{
    \draw[fill=black] (\n,0) circle (\c);
    }
    \draw[red,fill=red] (4,0) circle (\c);
    \draw[red,fill=red] (6,0) circle (\c);
    \draw[red,fill=red] (8,0) circle (\c);
    \draw (3,0) node {\tiny\encircled{1}}; 
    \foreach \n in {0,...,4}{
    \draw[fill=black] (9,0) ++(\a:\n) circle (\c);
    }
    \draw[red,fill=red] (9,0) ++(\a:2) circle (\c);
    \draw[red,fill=red] (9,0) ++(\a:4) circle (\c);
    \draw (9,0) ++ (\a:5) node {\tiny\encircled{2}};
    \foreach \n in {0,...,3}{
    \draw[fill=black] (9,0) ++(-\b:\n) circle (\c);
    }
    \draw[red,fill=red] (9,0) ++(-\b:2) circle (\c);
    \draw (9,0) ++ (-\b:4) node {\tiny\encircled{3}};
    \end{tikzpicture}
    
    \begin{tikzpicture}[scale=.55]
    \tikzmath{\a = 60;\b = 60;\c = .15;}
    \foreach \n in {4,...,9}{
    \draw[fill=black] (\n,0) circle (\c);
    }
    \draw[red,fill=red] (4,0) circle (\c);
    \draw[red,fill=red] (6,0) circle (\c);
    \draw[red,fill=red] (8,0) circle (\c);
    \draw (3,0) node {\tiny\encircled{1}};
    \fill[white] (9,0) circle (\c);
    \foreach \n in {0,...,4}{
    \draw[fill=black] (9,0) ++(\a:\n) circle (\c);
    }
    \draw[red,fill=red] (9,0) ++(\a:2) circle (\c);
    \draw[red,fill=red] (9,0) ++(\a:4) circle (\c);
    \draw (9,0) ++ (\a:5) node {\tiny\encircled{2}};
    \foreach \n in {0,...,3}{
    \draw[fill=black] (9,0) ++(-\b:\n) circle (\c);
    }
    \draw[red,fill=red] (9,0) ++(-\b:2) circle (\c);
    \draw (9,0) ++ (-\b:4) node {\tiny\encircled{3}};
    \draw[blue,fill=blue] (9,0) circle (\c);
    \draw[orange,fill=orange] (8,0) circle (\c);
    \draw[blue,fill=blue] (7,0) circle (\c);
    \draw[blue,fill=blue] (9,0) ++(\a:1) circle (\c);
    \draw[blue,fill=blue] (9,0) ++(-\b:1) circle (\c);
    \end{tikzpicture}
    
    \begin{tikzpicture}[scale=.55]
    \tikzmath{\a = 60;\b = 60;\c = .15;}
    \foreach \n in {4,...,8}{
    \draw[fill=black] (\n,0) circle (\c);
    }
    \draw[red,fill=red] (4,0) circle (\c);
    \draw[red,fill=red] (6,0) circle (\c);
    \draw (3,0) node {\tiny\encircled{1}};
    \foreach \n in {0,...,5}{
    \draw[fill=black] (8,0) ++(\a:\n) circle (\c);
    }
    \draw[red,fill=red] (8,0) ++(\a:3) circle (\c);
    \draw[red,fill=red] (8,0) ++(\a:5) circle (\c);
    \draw (8,0) ++ (\a:6) node {\tiny\encircled{2}};
    \foreach \n in {0,...,3}{
    \draw[fill=black] (8,0) ++(-\b:\n) circle (\c);
    }
    \draw[red,fill=red] (8,0) ++(-\b:2) circle (\c);
    \draw (8,0) ++ (-\b:4) node {\tiny\encircled{3}};
    \draw[blue,fill=blue] (8,0) circle (\c);
    \draw[orange,fill=orange] (8,0) ++(\a:1) circle (\c);
    \draw[blue,fill=blue] (7,0) circle (\c);
    \draw[blue,fill=blue] (8,0) ++(\a:2) circle (\c);
    \draw[blue,fill=blue] (8,0) ++(-\b:1) circle (\c);
    \end{tikzpicture}
    \end{multicols}
    
    \caption{A 3-star (left), a [3]-central substar (middle), and a central star permutation by a $120^\circ$ rotation (right).}
    \label{fig:3star}
\end{figure}
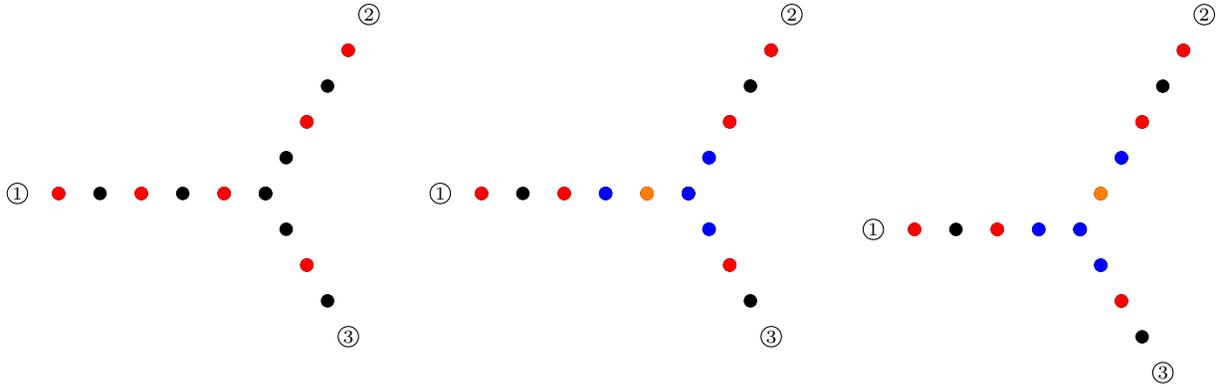

\end{example}

\begin{definition}
Given a star $s \in \Star_n(r)$ with strands labeled by $[n]=\{1,\ldots, n\}$, an \textit{$X$-substar} of $s$ is the identification, at the final vertex, of closing substrands of the strands labeled by $X \subseteq [n]$. An \textit{$X$-central substar} is an $X$-substar such that each strand has the same number of unmarked vertices, i.e., if it belongs to $\Star_{|X|}(r')$ for some $r' \leq r$.
\end{definition}

\begin{definition}
Let $s \in \Star_n(r)$ and $c$ be an $X$-central substar of $s$. A \textit{central star permutation} of $s$ by $c$ is an $n$-star obtained by some given permutation of the closing substrands in $c$.
\end{definition}

For a positive integer $k$ we use the notation $[k]$ to denote the set $\{1,\ldots,k\}$. The 3-star of Figure \ref{fig:3star} has a $[3]$-central substar and a $[2]$-central substar where the substars are embedded in the larger star and colored for distinction (blue = unmarked, orange = marked). Acting on the stands of the [3]-central substar by the permutation $(123)$ gives the star on the right in Figure \ref{fig:3star}. Acting on the stands of the $[2]$-central substar by the permutation $(12)$ gives the star on the right in Figure \ref{fig:another}.

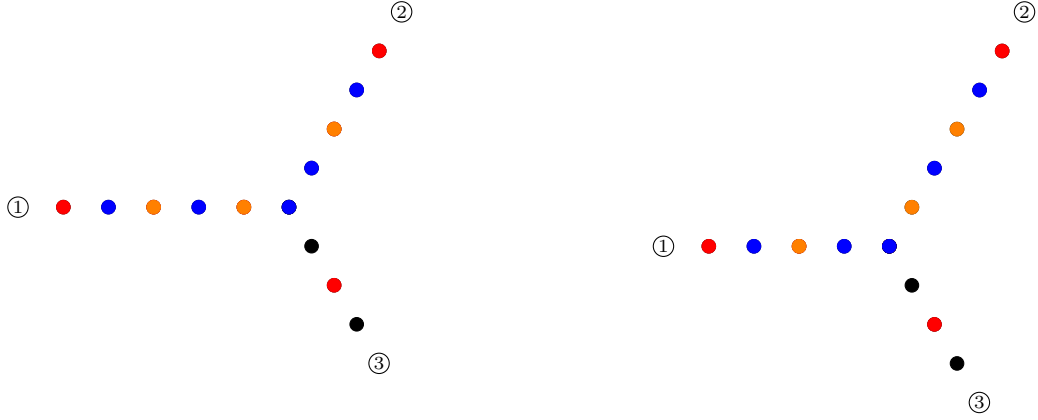
\begin{figure}[h]
    \centering
    \begin{multicols}{2}
    \begin{tikzpicture}[scale=.6]
    \tikzmath{\a=60;\b=60;\c=.15;}
    \foreach \n in {4,...,9}{
    \draw[fill=black] (\n,0) circle (\c);
    }
    \draw[red,fill=red] (4,0) circle (\c);
    \draw[red,fill=red] (6,0) circle (\c);
    \draw[red,fill=red] (8,0) circle (\c);
    \draw (3,0) node {\tiny\encircled{1}}; 
    \foreach \n in {0,...,4}{
    \draw[fill=black] (9,0) ++(\a:\n) circle (\c);
    }
    \draw[red,fill=red] (9,0) ++(\a:2) circle (\c);
    \draw[red,fill=red] (9,0) ++(\a:4) circle (\c);
    \draw (9,0) ++ (\a:5) node {\tiny\encircled{2}};
    \foreach \n in {0,...,3}{
    \draw[fill=black] (9,0) ++(-\b:\n) circle (\c);
    }
    \draw[red,fill=red] (9,0) ++(-\b:2) circle (\c);
    \draw (9,0) ++ (-\b:4) node {\tiny\encircled{3}};
    \draw[blue,fill=blue] (5,0) circle (\c);
    \draw[orange,fill=orange] (6,0) circle (\c);
    \draw[blue,fill=blue] (7,0) circle (\c);
    \draw[orange,fill=orange] (8,0) circle (\c);
    \draw[blue,fill=blue] (9,0) circle (\c);
    \draw[blue,fill=blue] (9,0) ++(\a:1) circle (\c);
    \draw[orange,fill=orange] (9,0) ++(\a:2) circle (\c);
    \draw[blue,fill=blue] (9,0) ++(\a:3) circle (\c);
    \end{tikzpicture}
    
\begin{tikzpicture}[scale=.6]
    \tikzmath{\a=60;\b=60;\c=.15;}
    \foreach \n in {4,...,8}{
    \draw[fill=black] (\n,0) circle (\c);
    }
    \draw[red,fill=red] (4,0) circle (\c);
    \draw[red,fill=red] (6,0) circle (\c);
    \draw[red,fill=red] (8,0) circle (\c);
    \draw (3,0) node {\tiny\encircled{1}}; 
    \foreach \n in {0,...,5}{
    \draw[fill=black] (8,0) ++(\a:\n) circle (\c);
    }
    \draw[red,fill=red] (8,0) ++(\a:5) circle (\c);
    \draw (8,0) ++ (\a:6) node {\tiny\encircled{2}};
    \foreach \n in {0,...,3}{
    \draw[fill=black] (8,0) ++(-\b:\n) circle (\c);
    }
    \draw[red,fill=red] (8,0) ++(-\b:2) circle (\c);
    \draw (8,0) ++ (-\b:4) node {\tiny\encircled{3}};
    \draw[blue,fill=blue] (5,0) circle (\c);
    \draw[orange,fill=orange] (6,0) circle (\c);
    \draw[blue,fill=blue] (7,0) circle (\c);
    \draw[blue,fill=blue] (8,0) circle (\c);
    \draw[blue,fill=blue] (8,0) ++(\a:2) circle (\c);
    \draw[orange,fill=orange] (8,0) ++(\a:1) circle (\c);
    \draw[orange,fill=orange] (8,0) ++(\a:3) circle (\c);
    \draw[blue,fill=blue] (8,0) ++(\a:4) circle (\c);
    \end{tikzpicture}
    \end{multicols}
    
    \caption{A [2]-central substar (left) of the 3-star in Figure \ref{fig:3star} and a central star permutation by (12) (right).}
    \label{fig:another}
\end{figure}

\begin{definition}
Two stars $s,s' \in \Star_n(r)$ are related by $\sim$, written $s \sim s'$, if there is a finite sequence $s = s_0, s_1, \ldots, s_k = s'$ of stars such that $s_i$ is a central star permutation of $s_{i-1}$ for $i = 1, \ldots, k$.
\end{definition}

\begin{theorem}\label{thrm:main}
If $n\geq 2$, then $P(n{\bf x}) = |\Star_n(r)/\sim|$.
\end{theorem}

\begin{proof}
We shall prove the correspondence from $n$-stars up to $\sim$ to partitions of $n\x = (n,\ldots,n)$ is bijective. First, it is surjective: given any partition of $nx$, we can construct an $n$-star which encodes this partition. We do this by first noting that we can extend the proof of Lemma \ref{split} so that an equivalent statement holds for $n\x$, allowing us to consider the partition of $n\x$ as $n$ partitions of $\x = (1,\ldots,1)$. Next we encode each partition of $\x$ as a strand. The identification of these $n$ stands is an $n$-star which corresponds to the original partition of $n\x$, which can be seen by simply reversing these steps.

Now, all that is left to do is prove the correspondence is injective. Let $s$ and $s'$ be two $n$-stars which correspond to the same partition of $n\x$. We will show $s \sim s'$, proceeding by induction on $n$. Note that when $n=2$, $s$ and $s'$ are bracelets and the statement corresponds to Theorem \ref{countbracelets}, which we prove above. Suppose $n > 2$ and assume that if two $(n-1)$-stars correspond to the same partition of $(n-1)\x$, then they are equivalent up to $\sim$.

Recall again, every stand of the $n$-stars $s$ and $s'$ give a partition of $\x$. For each $i = 1, \ldots, n$, let $V_i = \{v_{i1}, \ldots, v_{ik_i}\}$ be the partition of $\x$ encoded by the $i$th stand of $s$ where $k_i \in \mathbb{N}$, $v_{ij} \in \mathbb{Z}^n$ and $j = 1, \ldots, k_i$. Similarly, for each $i = 1, \ldots, n$, let $U_i = \{u_{i1}, \ldots, u_{i\ell_i}\}$ be the partition of $\x$ encoded by the $i$th stand of $s'$ where $\ell_i \in \mathbb{N}$, $u_{ij} \in \mathbb{Z}^n$ and $j = 1, \ldots, \ell_i$. Without loss of generality, order the sets $V_i$ and $U_i$ using the lexicographic order for each $i$. We shall identify $V_i$ with the $i$th stand of $s$ and $U_i$ with the $i$th strand of $s'$. Consider the $n$th strands $V_n$ and $U_n$. Note that if $V_n=U_n$, meaning $k_n = \ell_n$ and $v_{nj} = u_{nj}$ for all $j = 1,\ldots,k_n$, then removing the $n$th stands of both $s$ and $s'$ give two $(n-1)$-stars, denoted by $s \setminus V_n$ and $s' \setminus U_n$, respectively, corresponding to the same partition of $(n-1)\x$. By our hypothesis, there is a finite sequence $s \setminus V_n = s_0, \ldots, s_q = s' \setminus U_n$ such that $s_i$ is a central star permutation of $s_{i-1}$ for $i = 1, \ldots,q$. Attaching the $n$th stands, we get $s = (s \setminus V_n) \cup V_n = s_0 \cup V_n, \ldots, s_q \cup V_n = (s'\setminus U_n) \cup U_n = s'$. Thus, if we can show $V_n=U_n$ in the sense described above, then we have shown $s \sim s'$. 

To arrive at $V_n = U_n$, we will leave $U_n$ fixed and exchange the elements between the sets $V_n$ and $V_1, \ldots, V_{n-1}$ so that $V_n$ and $U_n$ match, where the exchange of elements is by central star permutations of $s$ so that the $n$th strands of $s'$ and a central star permutation of $s$ agree. Since $s$ and $s'$ encode the same partition of $n\x$, then $v_{i1} = u_{n1}$ for some $i$ since $V_i$ is written using the lexicographic order. Necessarily, $V_i \cup V_n$ is a central substar of $s$. Let $s_{1}$ be the star obtained from a central star permutation of $s$ by acting on $V_i \cup V_n$ by the permutation $(i\; n)$. Relabel the $V_i$ to correspond to the strands of $s_{1}$. Then $v_{n1}=u_{n1}$. Suppose, for some $t>1$, there is a sequence of central star permutations $s = s_0, \ldots, s_{t-1}$ such that, if the $V_i$ correspond to the strands of $s_{t-i}$, then $v_{nj} = u_{nj}$ for $j = 1, \ldots, t-1$. Since $s$ and $s'$ give the same partition, then $v_{ij} = u_{nt}$ for some $i,j$. Since $v_{n1} + \cdots + v_{n (t-1)} = u_{n1} + \cdots + u_{n (t-1)}$, then $v_{nt} + \cdots + v_{nk_n} = u_{nt} + \cdots + u_{n\ell_n}$. Since $v_{ij} = u_{nt}$, then $v_{ij} + \cdots + v_{ik_i} = u_{nt} + \cdots + u_{n\ell_n}$. 
Hence, $v_{nt} + \cdots + v_{nk_n} = v_{ij} + \cdots + v_{ik_i}$, so the vertices of $s_{t-1}$ corresponding to $\{v_{nt}, \ldots, v_{nk_n}\} \cup \{v_{ij}, \ldots, v_{ik_i}\}$ is a central substar of $s_{t-1}$. 
Acting on this central substar by $(i\; n)$ gives a new star $s_t$ and if we relabel the $V_i$ to correspond to the strands of $s_t$ then $v_{nj} = u_{nj}$ for $j = 1,\ldots,t$. Repeat these steps until $t=\ell_n$.
\end{proof}

\section{Multiplex Juggling Sequences} \label{sec: juggling}

In this section, we introduce and make a connection to objects that were first defined in \cite{Juggling} called  \textit{multiplex juggling sequences}. 
It is known that multiplex juggling sequences  count the number of $V_r$-combinations of $n{\bf x}$ \cite{Juggling}. Thus, there must be a way to connect $n$-stars (up to the relation induced by central star permutations) to  multiplex juggling sequences. Indeed, with some work, we can establish a correspondence, which we detail here.

Respecting the notation in \cite{Juggling}, vectors belonging to a juggling sequence are bold faced and, when describing the specific entries of a vector, angled brackets $\langle \cdot \rangle$ will be used. Furthermore, these vectors will ignore trailing zeros, for instance, the vector $\langle  3,0,1,0,0\rangle$ will be written as $\langle 3,0,1 \rangle$ instead.

\begin{definition}\label{def:mj sequence}
A \textit{multiplex juggling sequence} for $n$ balls is a tuple $S = ({\bf s_0}, {\bf s_1}, \ldots, {\bf s_r})$ where
\begin{enumerate}
    \item ${\bf s_i}$ is a nonnegative integer vector whose entries sum to $n$ for each $i = 0, 1, \ldots, r$, and
    \item if ${\bf s_i} = \langle s_1, \ldots, s_h \rangle$, then ${\bf s_{i+1}} = \langle s_2 + b_1, \ldots, s_h + b_{h-1}, b_h, \ldots, b_{h'} \rangle$ where the nonnegative integers $b_j$ satisfy $s_1 = \sum_{j=1}^{h'} b_j$.
\end{enumerate}

The vectors ${\bf s_0}$ and ${\bf s_r}$ are referred to as the \textit{initial} and \textit{final state} of the multiplex juggling sequence, respectively. A juggling sequence beginning and ending in the same state is said to be \textit{periodic}. In some instances, we invoke an integer parameter $m$, called the \textit{hand capacity}, to bound above the entries of the vectors in a juggling sequence. Note that if $m$ is larger than the entry sum of the vectors, then the hand capacity places no restriction on the juggling sequences.

The set of all multiplex juggling sequences of $n$ balls of length $r$ with initial state ${\bf a}$ and final state ${\bf b}$ is denoted JS$({\bf a},{\bf b},r)$  and its cardinality is denoted js$({\bf a},{\bf b},r)$; we note here that $n$ doesn't appear in the notation JS$({\bf a},{\bf b},r)$ or js$({\bf a},{\bf b},r)$ since $n$ can be inferred from either ${\bf a}$ or ${\bf b}$. The set of all multiplex juggling sequences of $n$ balls of length $r$ with initial state ${\bf a}$, final state ${\bf b}$, and hand capacity $m$ is denoted JS$({\bf a},{\bf b},r,m)$ and its cardinality is denoted js$({\bf a},{\bf b},r,m)$.
\end{definition}

\begin{example}
The sequence $S = (\langle 2 \rangle, \langle 1,1 \rangle, \langle 2 \rangle, \langle 1,1 \rangle, \langle 2 \rangle)$ is a juggling sequence belonging to the set JS$(\langle 2 \rangle, \langle 2 \rangle, 4)$. As Figure \ref{fig:jugseq1} depicts, a juggling sequence (a sequence of vectors) is drawn as a sequence of vertical conveyor belts (a conveyor for each vector). Each conveyor belt has buckets at discrete positions containing some number of balls. The $i$th entry of a vector encodes the number of balls in the $i$th bucket. From one time step to the next, every ball falls into the bucket immediately below except those in the bottom bucket, which are redistributed (these actions correspond to condition (2) in Definition \ref{def:mj sequence}).  In terms of juggling, we think about redistributing as throwing a ball upwards.

\begin{figure}[h]
    \centering
    \begin{tikzpicture}[scale = .5]
\foreach \m in {0,5,10,15,20}{
\draw (\m,0) -- (\m,4);
\foreach \n in {1,...,3}{
\draw (\m,\n) -- (\m+2,\n);
\draw (\m+2,\n) -- (\m+3,\n+1);
}
\filldraw[blue,fill=blue] (.5,1.5) circle (.2);
\filldraw[blue,fill=blue] (5.5,1.5) circle (.2);
\filldraw[blue,fill=blue] (10.5,1.5) circle (.2);
\filldraw[blue,fill=blue] (15.5,2.5) circle (.2);
\filldraw[blue,fill=blue] (20.5,1.5) circle (.2);
\filldraw[orange,fill=orange] (1.5,1.5) circle (.2);
\filldraw[orange,fill=orange] (5.5,2.5) circle (.2);
\filldraw[orange,fill=orange] (11.5,1.5) circle (.2);
\filldraw[orange,fill=orange] (15.5,1.5) circle (.2);
\filldraw[orange,fill=orange] (21.5,1.5) circle (.2);
}
\filldraw (1,-1) node {$\langle 2 \rangle$};
\filldraw (6,-1) node {$\langle 1,1 \rangle$};
\filldraw (11,-1) node {$\langle 2 \rangle$};
\filldraw (16,-1) node {$\langle 1,1 \rangle$};
\filldraw (21,-1) node {$\langle 2 \rangle$};
\end{tikzpicture}
    \caption{A pictorial representation of the juggling sequence $S = (\langle 2 \rangle, \langle 1,1 \rangle, \langle 2 \rangle, \langle 1,1 \rangle, \langle 2 \rangle)$.}
    \label{fig:jugseq1}
\end{figure}
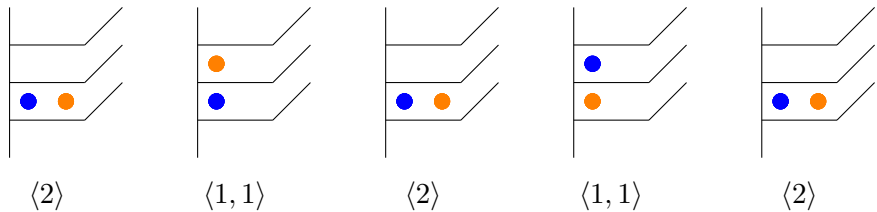

\end{example}

\begin{rmk} \label{rmk:indices}
We note that in \cite{Juggling}, Benedetti, Hanusa, Harris, Morales, and Simpson present a bijection between juggling sequences of length $r$ and Kostant's partition function in type $A_r$. Note that in this case the vectors being partitioned have length $r+1$. However, in what follows, given our convention to partition vectors of length $r$, we  give a connection between multiplex juggling sequences of length $r$ and partitions of vectors of the same length. This shift aligns the indices of the length of the juggling sequences with the length of the vectors in the partitions.
\end{rmk}

A juggling sequence of length $r$ beginning and ending in the state $\langle n \rangle$ encodes a partition of $n{\bf x}\in \mathbb{R}^{r}$ (Corollary 3.9, \cite{Juggling}). We can think of each such juggling sequence as being decomposed into $n$ juggling sequences of a single ball, namely, a juggling sequence beginning and ending in the state $\langle 1 \rangle$. A juggling sequence of a single ball encodes a partition of ${\bf x}$ by recording when the ball is thrown from the bottom bucket and how high the ball is thrown. For instance, the juggling sequence of the orange ball in Figure \ref{fig:jugseq1} is $(\langle 1 \rangle, \langle 0,1 \rangle, \langle 1 \rangle, \langle 1 \rangle, \langle 1 \rangle)$, so, pictorially, as shown in Figure~\ref{fig:orange and blue balls}, the orange ball starts in the bottom bucket, is thrown to height 2, falls down bucket, then at time step 2 is thrown to height~1 (the bottom bucket), and at time step 3 is thrown to height~1 again. Recall
$E_{i,j}=\sum_{k=i}^{j}e_k$. Following \cite{Juggling}, to obtain a partition of ${\bf x}$, let each throw of the ball correspond to a vector, in particular a throw at time $i$ to height $j$ corresponds to the vector $E_{i, i+j-1}$. Hence, the three throws of the orange ball encode the three vectors: $E_{1,1+2-1} = E_{1,2}$, $E_{3, 3+1-1} = E_3 = e_3$, and $E_{4,4+1-1} = E_4= e_4$. Thus, the trajectory of the orange ball encodes the partition $E_{1,2} + E_3 + E_4$, i.e. 12-3-4, which uses three vectors. Similarly, the juggling sequence of the blue ball is pictured in Figure \ref{fig:orange and blue balls}.  The juggling sequence of the blue ball encodes the partition $E_1 + E_2 + E_{34}$, i.e. 1-2-34. Combined, the juggling sequence of the blue and orange balls encode the following partition of $2{\bf x}$: $E_1 + E_2 + E_3 + E_4 + E_{12} + E_{34}$, i.e 1-12-3-34-4. Recall, the bracelet from Example~\ref{equivb4bracelets} encodes this same partition. This in fact is not a coincidence and we provide a bijection next.

\begin{figure}[h]
    \centering
    \begin{tikzpicture}[scale = .5]
\foreach \m in {0,5,10,15,20}{
\draw (\m,0) -- (\m,4);
\foreach \n in {1,...,3}{
\draw (\m,\n) -- (\m+2,\n);
\draw (\m+2,\n) -- (\m+3,\n+1);
}
\filldraw[orange,fill=orange] (.5,1.5) circle (.2);
\filldraw[orange,fill=orange] (10.5,1.5) circle (.2);
\filldraw[orange,fill=orange] (20.5,1.5) circle (.2);
\filldraw[orange,fill=orange] (5.5,2.5) circle (.2);
\filldraw[orange,fill=orange] (15.5,1.5) circle (.2);
}
\filldraw (1,-1) node {$\langle 1 \rangle$};
\filldraw (6,-1) node {$\langle 0,1 \rangle$};
\filldraw (11,-1) node {$\langle 1 \rangle$};
\filldraw (16,-1) node {$\langle 1 \rangle$};
\filldraw (21,-1) node {$\langle 1 \rangle$};

\tikzmath{\h=-7;}

\foreach \m in {0,5,10,15,20}{
\draw (\m,0+\h) -- (\m,4+\h);
\foreach \n in {1,...,3}{
\draw (\m,\n+\h) -- (\m+2,\n+\h);
\draw (\m+2,\n+\h) -- (\m+3,\n+1+\h);
}}
\filldraw[blue,fill=blue] (.5,1.5+\h) circle (.2);
\filldraw[blue,fill=blue] (5.5,1.5+\h) circle (.2);
\filldraw[blue,fill=blue] (10.5,1.5+\h) circle (.2);
\filldraw[blue,fill=blue] (15.5,2.5+\h) circle (.2);
\filldraw[blue,fill=blue] (20.5,1.5+\h) circle (.2);

\filldraw (1,-1+\h) node {$\langle 1 \rangle$};
\filldraw (6,-1+\h) node {$\langle 1 \rangle$};
\filldraw (11,-1+\h) node {$\langle 1 \rangle$};
\filldraw (16,-1+\h) node {$\langle 0,1 \rangle$};
\filldraw (21,-1+\h) node {$\langle 1 \rangle$};
\end{tikzpicture}
    \caption{The juggling sequence for the orange (above) and blue (below) balls from Figure \ref{fig:jugseq1}.}
    \label{fig:orange and blue balls}
\end{figure}
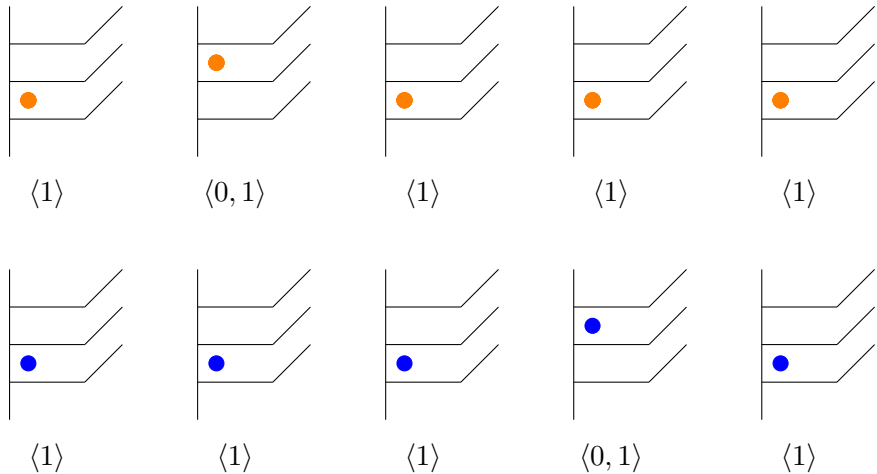

\subsection{Correspondence between juggling sequences and {\it n}-stars}

At this point, we can detail the correspondence between $n$-stars with $r-1$ unmarked vertices in each strand up to central star permutations and juggling sequences of length $r$ beginning and ending at the state $\langle n \rangle$, that is, a correspondence between the sets $\Star_n(r)/\sim$ and JS$(\langle n \rangle, \langle n \rangle, r)$. 
We know that such a correspondence exists since both encode a partition of $n{\bf x}\in\mathbb{R}^{r}$. Thus, an $n$-star up to $\sim$ corresponds to a juggling sequence in the sense that we can obtain the juggling sequence from the $n$-star by decoding its strands to get a partition of $n{\bf x}$, which can then be encoded into a juggling sequence. Since the correspondences from juggling sequences (of length $r$) with initial and final state $\langle n \rangle$ to partitions of $n{\bf x}$ and from $n$-stars up to $\sim$ to partitions of $n{\bf x}$ are bijective, then the decoding and encoding is unique and, hence, the correspondence from $n$-stars to juggling sequences with initial and final state $\langle n \rangle$ is indeed also bijective.

\begin{theorem}
If $n\geq 2$ and $r$ are positive integers, then js$(\langle n \rangle,\langle n \rangle,r) = |\Star_n(r)/\sim|$.
\end{theorem}

It is not necessary, however, to go through the partitions of $n{\bf x}$ in order to obtain a juggling sequence from an $n$-star, rather a more direct route exists. Let us detail this correspondence explicitly. Given an $n$-star, we can view each strand as encoding the juggling sequence of a single ball.  Specially, given a strand, add a single marked vertex and unmarked vertex, in that order, to the beginning of the strand, then the marked vertices of the modified strand encode the times when a throw occurs and the number of consecutive unmarked vertices between marked vertices encode the height the ball is thrown. Given any modified strand of this $n$-star, for the $i$th marked vertex, let $\ell_i$ be the number of unmarked vertices between the $i$th marked vertex and the $(i+1)$th marked vertex (or the end of the sequence if there are no remaining marked vertices).  Recall, the time and height of each throw  are all that is required to recover the juggling sequence for one ball. To go from a strand to a juggling sequence of a single ball, a throw is specified by the numbers $i$ and $\ell_i$: the $i$th maximal sequence of consecutive unmarked vertices encodes a throw to height $\ell_i$ at time step $0$ if $i=1$ or $\sum_{j=1}^{i-1} \ell_i$ if $i \neq 1$. From this, simply recover what remains of the juggling sequence for the modified strand by letting the ball fall to the bottom bucket:
\[
S = (\langle 1 \rangle,  {\bf e}_{\ell_1} ,  {\bf e}_{\ell_1-1}, \ldots, \langle 1 \rangle, {\bf e}_{\ell_2}, {\bf e}_{\ell_2-1}, \ldots, \langle 1 \rangle).
\]
The juggling sequence for the entire $n$-star is obtained by adding componentwise the sequences corresponding each strand. This concludes the exact correspondence between periodic juggling sequences of length $r$ with $n$ balls and $n$-stars with $r-1$ unmarked vertices in each strand. However, it is not immediately clear that this correspondence is well-defined with respect to the equivalence classes of $\Star_n(r)/\sim$. Thus, as a final remark, we  like to explain the consequences of central star permutations on the associated juggling sequence. 

A central star permutation exchanges inner strands of an $n$-star and effectively relabels the strands, since the strands of an $n$-star are labeled. To understand the corresponding action on a juggling sequence, color the balls in a juggling sequence; here, each color corresponds to a strand label in a $n$-star.  Thinking of a juggling sequence as one with colored balls (this is called a \textit{labeled juggling sequence}), a central star permutation translates to recoloring the balls when they are at the bottom bucket. Therefore, removing the colors of the balls from a labeled juggling sequence to transform into a unlabeled juggling has the same effect as identifying $n$-stars that are equivalent up to central star permutations.

\section*{Acknowledgements}
Elizabeth Gross was supported by the National Science Foundation (NSF), DMS-1945584. Pamela E.~Harris was supported by a Karen Uhlenbeck EDGE Fellowship.

\bibliographystyle{plain} 
\bibliography{references}{} 

\end{document}